\documentclass{article}
\usepackage{amsthm,amsmath,amssymb}
\usepackage{tikz}
\usepackage{tikz-cd}
\usepackage{graphicx}
\usetikzlibrary{arrows}
\usetikzlibrary{matrix,arrows,decorations.pathmorphing}
\usepackage{scalefnt}
\usepackage[]{algorithm2e}
\usepackage{tikz-qtree}
\usepackage{epigraph}
\usetikzlibrary{shapes,positioning,chains,fit,calc}
\usetikzlibrary{decorations.markings}
\tikzstyle arrowstyle=[scale=1]
\tikzstyle directed=[postaction={decorate,decoration={markings,
		mark=at position .65 with {\arrow[arrowstyle]{stealth}}}}]
\tikzstyle reverse directed=[postaction={decorate,decoration={markings,
		mark=at position .65 with {\arrowreversed[arrowstyle]{stealth};}}}]

\theoremstyle{theorem}
\newtheorem{theorem}{Theorem}
\newtheorem{corollary}[theorem]{Corollary}
\theoremstyle{definition}

\begin{document}

\title{What Moser \emph{Could} Have Asked: Counting Hamilton Cycles in Tournaments} 
\markright{Abbreviated Article Title}
\author{Neil J. Calkin,  Beth Novick  and  Hayato Ushijima-Mwesigwa}
\begingroup
\renewcommand{\thefootnote}{}
\footnotetext{%
	\emph{AMS Subject Classification: 05C20, 05C30, 05A16}		}%
\endgroup
\maketitle

\begin{abstract}
Moser asked for a construction of explicit tournaments on $n$ vertices having at least $(\frac{n}{3e})^n$ Hamilton cycles. We show that he could have asked for rather more.
\end{abstract}

\section{Introduction}
\epigraph{\dots the cycle has taken us up through forests.}{\textit{Robert M. Pirsig}}

In his classic book on tournaments, Moon \cite[Section 10]{Moon} discusses the question of exhibiting tournaments with a large number of Hamilton cycles. He poses the question (Exercise 4, attributed to Moser), of constructing a tournament on $n$ vertices having at least $(\frac{n}{3e})^n$ Hamilton cycles. Presumably, the intended construction is to take three tournaments, $T_1, T_2, T_3$, on $\frac{n}{3}$ vertices, and construct a new tournament $C_3(T_1, T_2, T_3)$ by orienting  all edges from $T_1$ to $T_2$, $T_2$ to $T_3$, and $T_3$ to $T_1$ (See Figure \ref{fig:T1} ).
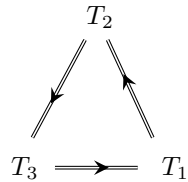
\begin{figure}[ht]
	\begin{center}
		\begin{tikzpicture} [scale = 1]
		\draw [directed, double] (0.8,1.7) ->(0.1,0.4) ;
		\draw  [directed, double] (0.4,0)--(1.5,0);
		\draw  [directed, double] (1.7,0.4)--(1.1,1.7);
		\node at (1,2){$T_{2}$};	 		
		\node at (0,0) {$T_{3}$};
		\node at (2,0) {$T_{1}$};
		\end{tikzpicture}
		\caption {$C_3(T_1, T_2, T_3)$}
		\label{fig:T1}
	\end{center}
\end{figure}
The number of Hamilton cycles in $C_3(T_1, T_2, T_3)$ which do not use any edges internal to $T_1, T_2,$ or $ T_3$ is
\begin{equation*}
 \frac{(\frac{n}{3})!^3}{\frac{n}{3}} \sim 
 \sqrt{\frac{8\pi^3 n}{3}}
  \left(\frac{n}{3e}\right)^n > \left(\frac{n}{3e}\right)^n.
\end{equation*}
 In this note, we show that this construction has many more Hamilton cycles. Indeed, if $T_1$, $T_2$, and $T_3$ are all transitive, we show that the number of Hamilton cycles is asymptotic to $\frac{1}{(1- \log 2)}\frac{(n-1)!}{(3 \log 2)^n}$.
\section{Background and Definitions}

A \emph{tournament} is an oriented, complete graph. A \emph{Hamilton cycle} or \emph{path} in a tournament $T$, is a spanning directed cycle or directed path in $T$. A tournament with no directed cycles is called \emph{transitive}

 Counting Hamilton paths and cycles in tournaments is a very old problem, dating back to the 1940's: in one of the first applications of the probabilistic method, Szele \cite{Szele} showed that the expected number of Hamilton paths in a random tournament is $\frac{n!}{2^{n-1}}$, therefore showing that there exists a tournament on $n$ vertices with at least this many Hamilton paths. The same argument shows that there exists a tournament with at least $\frac{(n-1)!}{2^{n}}$ Hamilton cycles. Moon observed that it seems difficult to give explicit tournaments with at least this many Hamilton cycles.
 
   Deep results of Cuckler \cite{cuckler} show that every regular tournament on $n$ vertices has at least $\frac{n!}{(2+o(1))^n}$ Hamilton cycles.
 
 Given tournaments $T_1, T_2, T_3$, we can construct a tournament $C_3(T_1, T_2, T_3)$ by orienting all edges from $T_1$ to $T_2$, $T_2$ to $T_3$, and $T_3$ to $T_1$. We will call such tournaments \emph{triangular}. Wormald \cite{wormald} showed that if $T_1, T_2, T_3$ are random tournaments, then the expected number of Hamilton cycles is $2\frac{(n-1)!}{2^{n}}$.
 
  We show that \emph{all} triangular tournaments have a relatively large number of Hamilton cycles, even in the extreme case when they constructed from transitive tournaments.
 
 Let $S(m,k)$ denote the Stirling number of the second kind, that is, $S(m,k)$ is the number of set partitions of $\{1,2, \dots m\}$ into exactly $k$ parts. 

\section{Main Result}
\begin{theorem} \label{thm:main}
	Let $T_1, T_2$, and $T_3$ be any tournaments on $m_1, m_2$, and $m_3$ vertices respectively. Then the number $H$ of Hamilton cycles in $C_3(T_1,T_2,T_3)$ is at least 
	\begin{equation}
H \geq	\sum_{k=1}^{\min \{ m_1, m_2, m_3\}}S(m_1,k)S(m_2,k)S(m_3,k) \frac{k!^3}{k},
	\label{thm1}
	\end{equation}
 with equality when $T_1, T_2$, and $T_3$ are transitive.
\end{theorem}
\begin{corollary}\label{cor:main}
	If $T_1, T_2$, and $T_3$ are transitive tournaments on $\frac{n}{3}$ vertices, the number of Hamilton cycles in $C_3(T_1,T_2,T_3)$ is asymptotic to
	\begin{equation}
	 \frac{ 1}{(1 - \log 2) }\frac{(n-1)!}{(3\log 2)^{n}} \simeq 3.25889 \frac{(n-1)!}{(2.07944)^n}.
	\end{equation} 
\end{corollary}	

\begin{proof}[Proof of Theorem \ref{thm:main}]
	Take any Hamilton cycle $C$, of $C_3(T_1,T_2,T_3)$, and consider $C$ restricted to  $T_{1}$, $T_{2}$, and $T_{3}$. Since a Hamilton cycle meets every vertex in $T_{1}$, $T_{2}$, and $T_{3}$ exactly once,  $C$ visits each subtournament the same number of times, say $k$. Hence for each $T_i$, $C$ will induce a collection of $k$ disjoint paths that cover the vertices of $T_{i}$. We will refer to such a collection of $k$ paths as a $k$-path cover. Similarly, given $k$-path covers for $T_1, T_2, T_3$, we can construct a Hamilton cycle by joining these $k$-path covers together. The  number of ways of doing this is $k!^3/k$. Thus, if $P(T_i,k)$ denotes the number of $k$-path covers of $T_{i}$, then the number of Hamilton cycles of $C_3(T_1, T_2, T_3)$ which induce $k$-path covers in $T_{1}$, $T_{2}$, and $T_{3}$ is 
	\begin{equation}
 P(T_1,k) P(T_2,k) P(T_3,k)\frac{k!^3}{k}.
	\end{equation}
	It follows that the number of Hamilton cycles in $C_3(T_1,T_2,T_3)$ is 
		\begin{equation*}
			\sum_{k=1}^{\min \{ m_1, m_2, m_3\}}P(T_1,k)P(T_2,k)P(T_3,k) \frac{k!^3}{k}.
		\end{equation*}
	For any set partition of the vertex set of $T_i$ into $k$ nonempty sets, each part will induce a subtournament of $T_i$. R\'edei \cite{Redei} showed that every tournament has a Hamilton path, thus each partition into $k$ sets will induce at least one $k$-path cover of $T_i$. Therefore the number of Hamilton cycles in $C_3(T_1,T_2,T_3)$ is at least 
	\begin{equation*}
		\sum_{k=1}^{\min \{ m_1, m_2, m_3\}}S(m_1,k)S(m_2,k)S(m_3,k) \frac{k!^3}{k}	
	\end{equation*}
	as claimed.
	
	In the case that each $T_i$ is transitive, each subtournament will have exactly one Hamilton path, hence we have equality in (\ref{thm1}).
		\end{proof}
			\begin{proof}[Proof of Corollary \ref{cor:main}]
		Suppose now that each $T_i$ is a transitive tournament on $m$ vertices, then the number of Hamilton cycles in $C_3(T_1,T_2,T_3)$ is equal to 
			\begin{equation}
				\sum_{k=1}^{m}S(m,k)^3 \frac{k!^3}{k}.
		\label{transitive_sum}
			\end{equation}

	As with many combinatorial sums, the summands in (\ref{transitive_sum}) are approximated rather well by a normal distribution. Indeed, if we let	
		\begin{equation*}
			\mu = \frac{1}{2\log 2} \text{  \quad and \quad }  \sigma = \frac{\sqrt{1- \log 2}}{2 \log 2},
		\end{equation*}
 define $f(m) = \sum_{k=1}^{m}S(m,k)k!$, and write  $p(m,k) = \frac{S(m,k)k!}{f(m)}$, then Bender \cite{Bender} shows that $p(m,k)$ is asymptotically normal with mean $\mu m $ and variance $ \sigma^2 m$. Hence, $p(m,k)^3$ is also proportional to a normal distribution, at least in a range of $k$ close to $\mu m$. This allows us to approximate the sum $\sum_{k=1}^{m}S(m,k)^3 \frac{k!^3}{k}$ by an integral, showing that 
 \begin{equation*}
 	\sum_{k=1}^{m}S(m,k)^3  \frac{k!^3}{k} \sim  f(m)^3 \frac{3^{\frac{1}{2}} 2^{\frac{1}{2}}  \pi^{\frac{1}{2}}}{3  \mu  \sigma^2 (2\pi)^{\frac{3}{2}} m^2}.
 \end{equation*}
From Wilf \cite[p. 176]{Wilf}, we know that
	\begin{equation*}
		f(m) \sim \frac{m!}{2(\log 2)^{m+1}},
	\end{equation*}
	Therefore with $n = 3m$, two applications of Stirling's approximation for $n!$ yields
		 \begin{eqnarray*}
		 	\sum_{k=1}^{m}S(m,k)^3  \frac{k!^3}{k} &\sim &  \bigg (\frac{m!}{2(\log 2)^{m+1}} \bigg )^3 \frac{3^{\frac{1}{2}} 2^{\frac{1}{2}}  \pi^{\frac{1}{2}}}{3  \mu  \sigma^2 (2\pi)^{\frac{3}{2}} m^2}\\	
 &\sim & \frac{ \sqrt{2 \pi n}}{n(1 - \log 2) }	\bigg(\frac{n}{3e \log 2}\bigg)^n\\	 	
 &\sim &		 	\frac{ 1}{(1 - \log 2) }\frac{(n-1)!}{(3\log 2)^{n}}\\
		 \end{eqnarray*}
		

\end{proof}
\noindent {\textbf{Acknowledgment.}}
	The authors are very grateful to Rod Canfield for helpful advice regarding the asymptotics.

\vfill\eject

\end{document}